\documentclass[12pt,reqno]{amsart}
\usepackage{amsmath, amsthm, amssymb}

\advance \topmargin by -\headheight
\advance \topmargin by -\headsep
     
\evensidemargin \oddsidemargin
\newtheorem{theorem}{Theorem}[section]
\newtheorem{lemma}[theorem]{Lemma}

\theoremstyle{definition}

\theoremstyle{remark}

\numberwithin{equation}{section}

\def\bfa{{\mathbf a}}
\def\bfb{{\mathbf b}}

\def\bfm{{\mathbf m}}

\def\calN{{\mathcal N}}

\def\calS{{\mathcal S}}

\def\dbF{{\mathbb F}}\def\dbN{{\mathbb N}}
\def\dbR{{\mathbb R}}

\def\grA{{\mathfrak A}}

\def\alp{{\alpha}} 
\def\bet{{\beta}}  
 
\def\del{{\delta}}

\def\tet{{\theta}}  
\def\kap{{\kappa}}

\def\ome{{\omega}} 

\def\eps{\varepsilon}

\def\le{\leqslant} \def\ge{\geqslant}

\begin{document}
\title[A low-energy decomposition theorem]{A low-energy decomposition theorem}
\author[Antal Balog]{Antal Balog$^*$}
\address{Alfr\'ed R\'enyi Institute of Mathematics, Re\'altanoda u. 13-15, H-1053 Budapest, 
Hungary.}
\email{balog@renyi.mta.hu}
\author[Trevor D. Wooley]{Trevor D. Wooley}
\address{School of Mathematics, University of Bristol, University Walk, Clifton, 
Bristol BS8 1TW, United Kingdom}
\email{matdw@bristol.ac.uk}
\thanks{$^*$Research supported by the Hungarian National Science Foundation Grants 
K104183 and K109789.}
\date{}
\begin{abstract} We prove that any finite set of real numbers can be split into two parts, 
one part being highly non-additive and the other highly non-multiplicative.
\end{abstract}
\maketitle

\section{Introduction} The Erd\H os-Szemer\'edi sum-product conjecture asserts that the 
additive structure of a finite set of real numbers should be essentially independent of its 
multiplicative structure. Given finite sets of real numbers $A$ and $B$, define the sum set 
and product set by
$$A+B=\{a+b:(a,b)\in A\times B\}\quad \text{and}\quad A\cdot B=\{ab:(a,b)\in 
A\times B\}.$$
Then, on writing $|\calS|$ for the cardinality of a set $\calS$, the conjecture of Erd\H os 
and Szemer\'edi (see the introduction of \cite{ES1983}) asserts that for any $\eps>0$ and 
for any sufficiently large finite set $A\subset \dbR$, one should have
$$\max\{|A+A|,|A\cdot A|\}\ge |A|^{2-\eps}.$$
The sharpest conclusion in this direction available in the published literature is due to 
Solymosi \cite[Corollary 2.2]{Sol2009}, and shows that
$$\max\{|A+A|,|A\cdot A|\}\ge \frac{|A|^{4/3}}{2\lceil \log |A|\rceil^{1/3}}.$$
This result has recently been improved by Konyagin and Shkredov, to the extent that the 
exponent $\frac{4}{3}$ may now be replaced by $\frac{4}{3}+c$, for any $c<1/20598$ 
(see \cite[Theorem 3]{KS2015} and the discussion concluding the latter paper).\par

As is well known, should the elements of $A$ be controlled by additive structure, then 
$|A+A|$ is small. Likewise, should $A$ be controlled by multiplicative structure, then 
$|A\cdot A|$ is small. The Erd\H os-Szemer\'edi conjecture expresses the belief that these 
two behaviours cannot be exhibited simultaneously.\par

A concrete measure of the additivity of a set is its additive energy
$$E_+(A)=\text{card}\{\bfa\in A^4:a_1+a_2=a_3+a_4\}.$$
Similarly, the multiplicativity of a set is measured by its multiplicative energy
$$E_\times(A)=\text{card}\{\bfa\in A^4:a_1a_2=a_3a_4\}.$$
One also has corresponding measures of the energy between two sets, namely
$$E_+(A,B)=\text{card}\{ (\bfa,\bfb)\in A^2\times B^2:a_1+b_1=a_2+b_2\}$$
and
$$E_\times (A,B)=\text{card}\{ (\bfa,\bfb)\in A^2\times B^2:a_1b_1=a_2b_2\}.$$

\par Writing
$$r_{A+B}(x)=\text{card}\{ (a,b)\in A\times B: a+b=x\}$$
and
$$r_{A\cdot B}(x)=\text{card}\{ (a,b)\in A\times B: ab=x\},$$
we see that
$$E_+(A)=\sum_{x\in A+A}r_{A+A}(x)^2\quad \text{and}\quad E_\times (A)=
\sum_{x\in A\cdot A}r_{A\cdot A}(x)^2.$$
It follows from Cauchy's inequality that
\begin{equation}\label{1.1}
|A|^4=\biggl( \sum_{x\in A+A}r_{A+A}(x)\biggr)^2\le E_+(A)|A+A|,
\end{equation}
so that, whenever $|A+A|$ is small, then $E_+(A)$ is big. In similar fashion, if $|A\cdot A|$ 
is small, then $E_\times(A)$ is necessarily big. Thus, one might na\"ively believe that the 
sum-product conjecture is manifested by the phenomenon that one or other of $E_+(A)$ 
and $E_\times (A)$ is always small. However, a moments' reflection reveals that such is 
certainly not the case, since the respective converses of the above observations are in 
general false. Thus, if $N\in \dbN$ and
\begin{equation}\label{1.2}
A=\{0,1,\ldots ,N-1\}\cup \{N,N^2,\ldots ,N^N\},
\end{equation}
then
$$\min\{ E_+(A),E_\times (A)\} \gg N^3\gg |A|^3.$$

\par In \S3, we show that any set $A$ can be split into two parts $B$ and $C$, having the 
property that both $E_+(B)$ and $E_\times (C)$ are small. Consequently, the na\"ive belief 
expressed above is obstructed only by examples closely related to that defined by 
(\ref{1.2}).

\begin{theorem}\label{theorem1.1}
Let $A$ be a finite subset of the real numbers. Then, with $\del=\frac{2}{33}$, there exist 
disjoint subsets $B$ and $C$ of $A$, with $A=B\cup C$,
$$\max\{ E_+(B),E_\times (C)\}\ll |A|^{3-\del}(\log |A|)^{1-\del}$$
and
$$\max\{ E_+(B,C),E_\times (B,C)\}\ll |A|^{3-\del/2}(\log |A|)^{(1-\del)/2}.$$
\end{theorem}

This theorem shows that any finite set of real numbers can be split into a highly 
non-additive part and a highly non-multiplicative part, and indeed, at least half of the set is 
either highly non-additive or highly non-multiplicative. Moreover, this decomposition has a 
doubly orthogonal flavour, the respective parts $B$ and $C$ being approximately orthogonal 
in terms both of their mutual additive energy, and also their mutual multiplicative energy.\par

It is tempting to conjecture that such decompositions should exist having the property that 
$\max\{ E_+(B),E_\times (C)\}\ll |A|^{2+\eps}$, for any $\eps>0$. By applying 
(\ref{1.1}) and its multiplicative analogue, such would imply the Erd\H os-Szemer\'edi 
conjecture in full. However, as we demonstrate in \S2, this tempting conjecture is 
over-ambitious. Let us describe the exponent $\bet$ as being a {\it permissible low-energy 
decomposition exponent} when, for each $\eps>0$ and for all sufficiently large finite 
subsets $A$ of $\dbR$, there exist disjoint sets $B$ and $C$, with $A=B\cup C$ and
$$\max \{ E_+(B),E_\times (C)\}\le |A|^{2+\bet+\eps}.$$

\begin{theorem}\label{theorem1.2}
The infimum $\kap$ of all permissible low-energy decomposition exponents satisfies 
$\tfrac{1}{3}\le \kap\le \tfrac{31}{33}$.
\end{theorem}

In particular, there exist arbitrarily large finite subsets $A$ of $\dbR$ for which every 
decomposition into two parts $B$ and $C$ satisfies the lower bound
$$\max\{ E_+(B),E_\times(C)\}\gg |A|^{7/3}.$$
The problem of determining the infimal exponent $\kap$ seems interesting, as well as 
very delicate, and we do not have a reasonable conjecture as to its value.\par

Our methods extend naturally to other settings, with obvious adjustments to our previous 
definitions concerning additive and multiplicative energies, and associated concepts. For 
example, an analogous argument yields a related conclusion in the setting of the finite field 
$\dbF_p$ having $p$ elements. This we establish in \S4. 

\begin{theorem}\label{theorem1.3} Let $p$ be a large prime, and suppose that 
$A\subseteq \dbF_p$ satisfies $|A|\le p^\alp(\log p)^\bet$, where we write
$$\alp=\frac{101}{161}\quad \text{and}\quad \bet=\frac{71}{161}.$$
Then, with $\del=4/101$, there exist disjoint subsets $B$ and $C$ of $A$, satisfying 
$A=B\cup C$ and
$$\max\{E_+(B),E_\times(C)\}\ll |A|^{3-\del}(\log |A|)^{1-\del/2}.$$
When $|A|\ge p^\alp(\log p)^\bet$, meanwhile, one has instead
$$\max\{E_+(B),E_\times(C)\}\ll |A|^3(|A|/p)^{1/15}(\log |A|)^{14/15}.$$
\end{theorem}

In the second of these conclusions, the upper bound for $\max\{E_+(B),E_\times(C)\}$ 
becomes non-trivial only when $|A|$ is smaller than about $p(\log p)^{-14}$. It may be 
worth emphasising that a bound here uniform in $p$ is certainly not available, for a simple 
argument presented at the end of \S4 confirms that in the setting of the finite field 
$\dbF_p$, one always has
\begin{equation}\label{1.x}
E_+(B)\ge |B|^4/p\quad \text{and}\quad E_\times(C)\ge |C|^4/p,
\end{equation}
whence
$$\max\{E_+(B),E_\times(C)\}\gg |A|^3(|A|/p).$$

We mention a prototype application for the low-energy decomposition theorem recorded in 
Theorem \ref{theorem1.3}. In his Ph.D. thesis at the University of Toronto (see 
\cite[\S4]{Han2015}), Brendan Hanson gives a non-trivial bound for the character sum
\begin{equation}\label{1.3}
\sum_{a\in A}\sum_{b\in B}\sum_{c\in C}\sum_{d\in D} \chi(a+b+cd),
\end{equation}
where $\chi$ is a non-principal character modulo $p$, and $A, B, C, D$ are subsets of the 
finite field $\dbF_p$. All four sets can be somewhat smaller than $\sqrt{p}$ in his work, 
and hence he breaks the ``square-root barrier''. Hanson makes use of different arguments 
according to whether $E_+(C)$ or $E_\times(C)$ is small. Such an argument would 
naturally utilise a conclusion of the shape recorded in Theorem \ref{theorem1.3}: the sum 
(\ref{1.3}) can be split into two sums by writing $C=C_1\cup C_2$, with $E_+(C_1)$ and 
$E_\times(C_2)$ both small, and then each sum may be estimated in turn by appeal to 
one or other of Hanson's arguments.\par

It is not hard to extend our results from the above notions of energy to the analogous 
concept of $k$-fold energy. Given finite subsets $A_i$ of real numbers, define  
\begin{align*}
E_+(A_1,\dots,A_k)&=\text{card}\{\bfa,\bfa'\in A_1\times \ldots \times 
A_k:a_1+\dots+a_k=a'_1+\dots+a'_k\},\\
E_\times(A_1,\dots,A_k)&=\text{card}\{\bfa,\bfa'\in A_1\times \ldots \times A_k:
a_1\cdots a_k=a_1'\cdots a_k'\}.
\end{align*}
For the sake of convenience, we then put
$$E^{(k)}_+(A)=E_+(\underbrace{A,\dots,A}_k)\quad \text{and}\quad 
E^{(k)}_\times (A)=E_\times(\underbrace{A,\dots,A}_k).$$
As an immediate consequence of Theorem \ref{theorem1.1}, in \S5 we obtain the following 
low-energy decomposition theorem for $k$-fold energies.

\begin{theorem}\label{theorem1.4}
Let $A$ be a finite subset of the real numbers, and suppose that $m$ and $n$ are integers 
with $m\ge 2$ and $n\ge 2$. Then, with $\del=\frac{2}{33}$, there exist disjoint subsets 
$B$ and $C$ of $A$, with $A=B\cup C$, 
$$E^{(m)}_+(B)\ll |A|^{2m-1-\del}(\log |A|)^{1-\del}$$
and
$$E^{(n)}_\times(C)\ll |A|^{2n-1-\del}(\log |A|)^{1-\del}.$$
\end{theorem}

Our approach to proving this theorem involves a reduction to the $2$-fold energy central to 
Theorem \ref{theorem1.1}, and fails to make any use of the richer structure available for the 
$k$-fold energy. It seems unlikely that this approach is particularly effective. Rather, we 
simply want to point out that low-energy decomposition theorems are available for 
$k$-fold energies. One would like to see a much sharper result, involving a saving in the 
exponent which grows with $m$ (or $n$) in place of the constant saving $\frac{2}{33}$ in 
the conclusion of Theorem \ref{theorem1.4}. We note in this context that a construction 
analogous to that in \S2 delivering Theorem \ref{theorem1.2} shows only that there exist 
arbitrarily large finite subsets $A$ of $\dbR$ for which, for all natural numbers $m$ and $n$ 
with $m\ge 2$ and $n\ge 2$, and for every decomposition of $A$ into two parts $B$ and 
$C$, one has either 
$$E^{(m)}_+(B)\gg |A|^{(4m-1)/3}\quad \text{or}\quad E^{(n)}_\times(C)\gg 
|A|^{(4n-1)/3}.$$

\par Throughout this paper, we write $\lceil \tet\rceil$ for the smallest integer no smaller than 
$\tet$, and $\lfloor \tet\rfloor$ for the largest integer not exceeding $\tet$. Also, when 
describing ranges for integers in the definitions of sets, for example, we write $n\le N$ to 
denote the constraint $1\le n\le N$.\par

The authors wish to express their gratitude to Misha Rudnev for an insightful suggestion 
relevant to low-energy decompositions in finite fields. His suggestion of the use of a relation 
of the shape $E_+(A)\ll |A\cdot A|^{3/2}|A|$ led us to formulate Lemma \ref{lemma4.6}, 
and fostered the significant improvement of our previous bounds in Theorem \ref{theorem4.2} 
now recorded in Theorem \ref{theorem1.3}. We thank him for his generous contribution to this paper, 
and also for his comments on an earlier draft of this paper.

\section{Permissible low-energy decomposition exponents} We begin our exploration of 
low-energy decompositions with the proof of Theorem \ref{theorem1.2}. Here, we 
temporarily assume the truth of Theorem \ref{theorem1.1}, deferring its proof to \S3. 
Thus, we may suppose that, for each positive number $\eps$, there exists a permissible 
low-energy decomposition exponent $\bet$ with $\bet \le \frac{31}{33}+\eps$. It follows 
that the infimum $\kap$ of all such exponents satisfies $\kap\le \frac{31}{33}$. 
Consequently, it remains only to show that $\kap\ge \frac{1}{3}$.\par

Fix a large positive integer $N$, and put
$$A=\{(2m-1)2^n: \text{$m\le N^{2/3}$ and $n\le N^{1/3}$}\}.$$
Thus, one has $|A|=N+O(N^{2/3})$. We claim that whenever $B\subseteq A$ and 
$|B|\ge |A|/2$, then one necessarily has both
\begin{equation}\label{2.1}
E_+(B)\gg N^{7/3}\quad \text{and}\quad E_\times(B)\gg N^{7/3}.
\end{equation}

\par Suppose that $B\subseteq A$ and $|B|\ge |A|/2$. We first examine the multiplicative 
energy $E_\times(B)$. Note that
$$B\cdot B\subseteq \{(2m-1)2^n:\text{$m\le 2N^{4/3}$ and $n\le 2N^{1/3}$}\},$$
so that $|B\cdot B|\le 4N^{5/3}$. We therefore deduce from Cauchy's inequality that
$$(N/2)^4\le |B|^4=\biggl(\sum_{x\in B\cdot B}r_{B\cdot B}(x)\biggr)^2\le |B\cdot B|
\sum_{x\in B\cdot B}r_{B\cdot B}(x)^2\le 4N^{5/3}E_\times(B),$$
whence $E_\times(B)\geq N^{7/3}/2^6$.

Our discussion of the corresponding additive energy $E_+(B)$ entails more effort. For a fixed 
natural number $n$, let
$$M_n=\{m\in \dbN: \text{$m\le N^{2/3}$ and $(2m-1)2^n\in B$}\}.$$
Also, define
$$\calN=\{ n\in \dbN:|M_n|\ge \tfrac{1}{4}N^{2/3}\}.$$
Then it follows by means of an obvious averaging argument that 
$$|\calN|>\tfrac{1}{3}N^{1/3}+O(1).$$
Indeed, one has
$$\tfrac{1}{2}N+O(N^{2/3})<|B|=\sum_{n\in \dbN}|M_n|\le N^{2/3}|\calN|+
\tfrac{1}{4}N^{2/3}(N^{1/3}- |\calN|),$$
from which the desired conclusion follows. Note also that for all natural numbers $n$, the 
sumset $M_n+M_n$ is a subset of the natural numbers not exceeding $2N^{2/3}$, and 
hence $|M_n+M_n|\leq 2N^{2/3}$. It therefore follows from Cauchy's inequality, much as 
before, that for any fixed $n\in \dbN$ we have
$$|M_n|^4\le |M_n+M_n|E_+(M_n),$$
and we arrive at the lower bound
$$E_+(M_n)\ge \left( \tfrac{1}{4}N^{2/3}\right)^4/(2N^{2/3})=2^{-9}N^2\quad 
(n\in \calN).$$
We thus conclude that
\begin{align*}
E_+(B)&\ge \sum_{n\in\calN}\text{card}\biggl\{\bfm\in M_n^4: 
\sum_{i=1}^4(-1)^{i-1}(2m_i-1)2^n=0\biggr\}\\
&=\sum_{n\in\calN}E_+(M_n)\ge \tfrac{1}{3}2^{-9}N^{7/3}+O(N^2).
\end{align*}

By combining these conclusions, we see that when $B\subseteq A$ and $|B|\ge |A|/2$, then 
one necessarily has both of the lower bounds (\ref{2.1}), confirming our opening claim. In 
particular, whenever $\bet<\tfrac{1}{3}$, there exist arbitrarily large finite subsets $A$ of 
$\dbR$ for which, for some positive number $\eps$, every decomposition into two parts $B$ 
and $C$ satisfies the lower bound 
$\max \{ E_+(B),E_\times(C)\}\gg |A|^{2+\bet+\eps}$. Consequently, the infimum $\kap$ of all permissible low-energy decomposition exponents 
satisfies $\kap\ge \tfrac{1}{3}$. This completes our proof of Theorem \ref{theorem1.2}.

\section{Low-energy decompositions} Our goal in this section is the proof of the low-energy decomposition theorem recorded in Theorem \ref{theorem1.1}. The key ingredient in 
our proof of the latter is a version of the Balog-Szemer\'edi-Gowers lemma, which gives a 
quantitative version of the assertion that when $E_+(B)$ is large, then there exists a large 
subset $A'$ of $B$ for which $|A'+A'|$ is small (see \cite[Theorem 2]{Bal2007}). In order to 
extract the sharpest accessible conclusion, we apply the Balog-Szemer\'edi-Gowers lemma in 
the form given by Schoen (see \cite[Theorem 1.2]{Sch2014}). 

\begin{lemma}\label{lemma3.1} Let $B$ be a non-empty finite subset of an abelian group 
with $E_+(B)=\alpha |B|^3$. Then there exist subsets $A'$ and $B'$ of $B$ such that
$$\min\left\{|A'|,|B'|\right\}\gg \alpha^{3/4}\left(\log(1/\alpha)\right)^{-5/4}|B|$$
and
$$|A'-B'|\ll \alpha^{-7/2}\left(\log(1/\alpha)\right)^{5/2}\left(|A'||B'|\right)^{1/2}.$$
\end{lemma}

In order to proceed further, we must modify the conclusion of this lemma so that it 
supplies a bound for $|A'+A'|$. This we achieve through a consequence of Pl\"unnecke's 
inequality.

\begin{lemma}\label{lemma3.2} Let $A$ and $B$ be non-empty finite subsets of an abelian 
group. Then we have $|A+A|\leq |A+B|^2/|B|$.
\end{lemma}

\begin{proof} Let $A$ and $B$ be finite non-empty sets in an abelian group. Define $\alp$ 
by means of the relation $|A+B|=\alp |B|$. Then it follows from Pl\"unnecke's inequality, 
in the form given in the recent work of Petridis \cite[Theorem 1.1]{Pet2012}, that there 
exists a non-empty subset $X$ of $B$ such that $|X+hA|\le \alp^h|X|$. We apply this 
estimate in the special case $h=2$. We find that there exists a non-empty subset $X$ of 
$B$ such that
$$|A+A|\le |X+2A|\le \left( |A+B|/|B|\right)^2|X|\le \left( |A+B|/|B|\right)^2|B|,$$
and the desired conclusion follows.
\end{proof}

By combining Lemmata \ref{lemma3.1} and \ref{lemma3.2}, we obtain a version of the 
Balog-Szemer\'edi-Gowers lemma suitable for our application.

\begin{lemma}\label{lemma3.3} Let $G$ be an abelian group. Then there are positive 
constants $c_1$ and $c_2$, depending at most on $G$, with the following property. 
Suppose that $N$ is a sufficiently large natural number. Let $B$ be a non-empty finite 
subset of $G$ with $|B|\le N$, and suppose that $\del$ and $\tet$ are real numbers with 
$0<\del<1$. Then, either $E_+(B)\le N^{3-\delta}(\log N)^\tet$, or else there exists a 
subset $A'$ of $B$ such that
$$|A'|\ge c_1N^{1-3\delta/4}(\log N)^{(3\tet-5)/4}\quad \text{and}\quad |A'+A'|\le c_2
N^{7\delta}(\log N)^{5-7\tet}|A'|.$$
\end{lemma}

\begin{proof} Suppose that $B$, $\del$ and $\tet$ satisfy the hypotheses of the statement 
of the lemma. If $E_+(B)\le N^{3-\del}(\log N)^\tet$, then there is nothing to prove, so 
we may suppose that $E_+(B)>N^{3-\del}(\log N)^\tet$. Put $\alp=E_+(B)/|B|^3$. Then 
we have $$\alpha>N^{3-\delta}(\log N)^\tet/|B|^3\ge N^{-\delta}(\log N)^\tet,$$
and we deduce from Lemma \ref{lemma3.1} that there exist subsets $A'$ and $B'$ of $B$ 
with
\begin{align*}
\min\left\{ |A'|,|B'|\right\}&\gg \left( N^{3-\del}(\log N)^\tet /|B|^3\right)^{3/4}\left(\log 
N\right)^{-5/4}|B|\\
&\gg N^{1-3\del/4}(\log N)^{(3\tet-5)/4}
\end{align*}
and
$$|A'-B'|\ll \alp^{-7/2}\left( \log (1/\alp)\right)^{5/2}\left( |A'||B'|\right)^{1/2}.$$
However, Lemma \ref{lemma3.2} leads from the latter bound to the relation
$$|A'+A'|\le |A'-B'|^2/|B'|\ll \alpha^{-7}\left(\log(1/\alp)\right)^5|A'|\ll N^{7\delta}
(\log N)^{5-7\tet}|A'|.$$
The conclusion of the lemma now follows.
\end{proof}

We also need the fact that the multiplicative energy between two sets exhibits some 
subadditive behaviour. This is folklore, and follows as a simple consequence of Cauchy's 
inequality.

\begin{lemma}\label{lemma3.4} Let $A_j$ $(1\le j\le J)$ and $B_k$ $(1\le k\le K)$ be finite 
subsets of a ring. Then one has
$$E_\times\biggl( \bigcup_{j=1}^JA_j,\bigcup_{k=1}^KB_k\biggr)\le 
JK\sum_{j=1}^J\sum_{k=1}^KE_\times(A_j,B_k).$$
\end{lemma}

Finally, we recall a generalisation of the key ingredient of Solymosi \cite{Sol2009} in proving 
his sum-product estimate. 

\begin{lemma}\label{lemma3.5} Let $A$ and $B$ be non-empty finite subsets of the real 
numbers. Then
$$E_\times(A,B)\le 4|A+A|\cdot |B+B|\lceil \log(\min\{|A|,|B|\})\rceil .$$ 
\end{lemma}

\begin{proof} The desired conclusion follows from the remarks concluding 
\cite[\S2.3]{Sol2009}. See also \cite[Theorem 6]{KS2015}.
\end{proof}

We are now equipped for the main act of this section.

\begin{proof}[The proof of Theorem \ref{theorem1.1}] Let $\del$ and $\tet$ be 
parameters with $0<\del<1$ to be fixed in due course, and let $c_1$ and $c_2$ be the 
positive constants whose existence is guaranteed via Lemma \ref{lemma3.3}. We consider 
a subset $A$ of the real numbers with $|A|=N$, where $N$ is a sufficiently large natural 
number. We prove first that there exist disjoint subsets $B$ and $C$ of $A$, with 
$A=B\cup C$ and
\begin{equation}\label{4.a}
\max\{ E_+(B),E_\times (C)\}\ll |A|^{3-\del}(\log |A|)^\tet.
\end{equation}

\par Should one have $E_+(A)\le N^{3-\delta}(\log N)^\tet$, then $A=A\cup\emptyset$ is 
trivially a decomposition of the type we seek, and so we may suppose henceforth that 
$E_+(A)>N^{3-\delta}(\log N)^\tet$. We now proceed inductively to define certain subsets 
$A_j$ of $A$ for $1\le j\le K$, for a suitable integer $K$. Suppose that $k\ge 0$ and that 
the first $k$ of these sets have been defined. We put
\begin{equation}\label{4.0}
C_k=\bigcup_{j=1}^kA_j\quad \text{and}\quad B_k=A\setminus C_k.
\end{equation}
Should $E_+(B_k)\le N^{3-\delta}(\log N)^\tet$, then we set $K=k$ and stop. Otherwise, 
we define the set $A_{k+1}$ as follows. We may suppose that 
$E_+(B_k)>N^{3-\delta}(\log N)^\tet$, and so it follows from Lemma \ref{lemma3.3} that 
there exists $A_{k+1}\subseteq B_k\subseteq A$ such that
\begin{equation}\label{4.1}
|A_{k+1}|\ge c_1N^{1-3\del/4}(\log N)^{(3\tet-5)/4}
\end{equation}
and
\begin{equation}\label{4.2} 
|A_{k+1}+A_{k+1}|\le c_2N^{7\del}(\log N)^{5-7\tet}|A_{k+1}|.
\end{equation}
Having defined the set $A_{k+1}$, we may define $B_{k+1}$ and $C_{k+1}$ according 
to (\ref{4.0}), and repeat this decomposition argument.\par

The iteration described in the last paragraph must terminate for a value of $K$ satisfying 
$K\le K_0$, where $K_0=\lfloor c_1^{-1}N^{3\delta/4}(\log N)^{(5-3\tet)/4}\rfloor $. For
$$|B_k|=|A|-\sum_{j=1}^k|A_j|\le N-kc_1N^{1-3\del/4}(\log N)^{(3\tet-5)/4},$$
whence $B_{K_0}$ (if it exists) must satisfy 
$|B_{K_0}|\le c_1N^{1-3\del/4}(\log N)^{(3\tet-5)/4}$. In such circumstances, a trivial 
estimate yields the bound
$$E_+(B_{K_0})\le \left( c_1N^{1-3\del/4}(\log N)^{(3\tet-5)/4}\right)^3\le N^{3-\del},
$$
and our iteration stops.\par

Now equipped with the sets $A_1,\ldots ,A_K$ defined by this iterative process, we ease 
our exposition by abbreviating $B_K$ to $B$ and $C_K$ to $C$. Note that $A$ is the 
disjoint union of $B$ and $C$. The first observation is that a defining feature of our 
iteration is the bound $E_+(B)\le N^{3-\del}(\log N)^\tet$. We group the subsets $A_j$ 
by cardinality, taking $\grA_m$ to be the union of those subsets $A_j$ for which 
$2^{-m}N<|A_j|\le 2^{1-m}N$. Notice that cardinality constraints ensure that each set 
$\grA_m$ consists of no more than $2^m$ of the subsets $A_j$, and, moreover, one has 
$m=O(\log N)$. In the first instance, grouping together the sets $A_j$ in this way leads via 
Lemma \ref{lemma3.4} to the bound
\begin{align}
E_\times(C)&=E_\times (C,C)\ll (\log N)^2\sum_{k,l}E_\times(\grA_k,\grA_l)\notag \\
&\ll (\log N)^2\sum_k\sum_{A_i\subseteq \grA_k}\sum_l
\sum_{A_j\subseteq \grA_l}2^{k+l}E_\times(A_i,A_j).\label{4.wa}
\end{align}
Next, an application of Lemma \ref{lemma3.5} propels us to the estimate
$$E_\times (C)\ll (\log N)^3\biggl( \sum_k\sum_{A_i\subseteq \grA_k}2^k|A_i+A_i|
\biggr)^2.$$
Consequently, on applying the property (\ref{4.2}) of these subsets $A_i$, we infer that
$$E_\times (C)\ll N^{14\del}(\log N)^{13-14\tet}\biggl( \sum_k
\sum_{A_i\subseteq \grA_k}2^k|A_i|\biggr)^2.$$
But the property (\ref{4.1}) of the subsets $A_i\subseteq \grA_k$ ensures that the inner 
sum here is empty whenever
\begin{equation}\label{4.ww2}
2^{1-k}N<c_1N^{1-3\delta/4}(\log N)^{(3\tet-5)/4}.
\end{equation}
Moreover, it is apparent that for each $k$ one has
\begin{equation}\label{4.ww1}
\sum_{A_i\subseteq \grA_k}|A_i|\le N.
\end{equation}
Thus we deduce that
\begin{align*}
E_\times (C)&\ll N^{2+14\del}(\log N)^{13-14\tet}\biggl( N^{3\delta/4}
(\log N)^{(5-3\tet)/4}\biggr)^2\\
&\ll N^{2+31\del/2}(\log N)^{31(1-\tet)/2}.
\end{align*}

We now set $3-\del=2+31\del/2$ and $\tet=31(1-\tet)/2$, which is to say 
$\del=\tfrac{2}{33}$ and $\tet=\tfrac{31}{33}$, and conclude that
\begin{equation}\label{4.5}
E_+(B)\le N^{3-\del}(\log N)^\tet\quad \text{and}\quad 
E_\times (C)\ll N^{3-\del}(\log N)^\tet .
\end{equation}
Since $N=|A|$, this confirms the relations (\ref{4.a}). In order to complete the proof of 
Theorem \ref{theorem1.1}, it now remains only to establish that
$$\max\{ E_+(B,C),E_\times(B,C)\}\ll N^{3-\del/2}(\log N)^{\tet/2}.$$  
But rewriting the energy between the sets $B$ and $C$ in the form
$$E_+(B,C)=\sum_{x\in B+C}r_{B+C}(x)^2=\sum_{y\in (B-B)\cap (C-C)}r_{B-B}(y)
r_{C-C}(y),$$
we infer from Cauchy's inequality that
\begin{align*}
E_+(B,C)&\le \biggl(\sum_{y\in B-B}r_{B-B}(y)^2\biggr)^{1/2}\biggl( \sum_{y\in C-C}
r_{C-C}(y)^2\biggr)^{1/2}\\
&=E_+(B)^{1/2}E_+(C)^{1/2}.
\end{align*}
Thus we conclude from (\ref{4.5}) and the trivial estimate $E_+(C)\ll N^3$ that
$$E_+(B,C)\ll \left(N^{3-\del}(\log N)^\tet \right)^{1/2}(N^3)^{1/2}
\ll N^{3-\del/2}(\log N)^{\tet/2}.$$
By an entirely analogous argument, one finds also that
$$E_\times(B,C)\ll N^{3-\del/2}(\log N)^{\tet/2}.$$
This completes the proof of the final claim of Theorem \ref{theorem1.1}.
\end{proof}

\section{Low-energy decompositions in finite fields}
The strategy prosecuted in \S3 may be adapted without serious difficulty to the setting of 
finite fields $\dbF_p$, with $p$ prime. The only ingredient which requires serious 
modification is Lemma \ref{lemma3.5}, which bounds the multiplicative energy of a set in 
terms of its sumset. One approach to handling this difficulty is by appeal to work of 
Bourgain \cite{Bou2010} and Rudnev \cite{Rud2012}.

\begin{lemma}\label{lemma2dash} Suppose that $A\subset\dbF_p$. Then
$$E_\times (A)\ll |A+A|^{9/4}|A|^{1/2}+p^{-1/4}|A+A|^2|A|^{5/4}.$$
Moreover, provided that $|A|<\sqrt p$, one has the sharper bound
$$E_\times(A)\ll |A+A|^{7/4}|A| \lceil \log|A|\rceil .$$
\end{lemma}

\begin{proof} The first bound on the multiplicative energy contained in this lemma is simply 
\cite[Proposition 1]{Bou2010}, whilst the second is essentially \cite[equation (3.22)]
{Rud2012}.
\end{proof}

By following the path described in \S3 leading to the proof of Theorem \ref{theorem1.1}, 
the reader will have little difficulty in obtaining the conclusion recorded in the following 
theorem.

\begin{theorem}\label{theorem4.2} Let $p$ be a large prime, and suppose that 
$A\subseteq \dbF_p$ satisfies $|A|<\sqrt{p}$. Then, with $\del=4/227$, there exist 
disjoint subsets $B$ and $C$ of $A$, with $A=B\cup C$ and
$$\max\{E_+(B),E_\times(C)\}\ll |A|^{3-\del}(\log |A|)^{1+\del/2}.$$
When $|A|\ge \sqrt{p}$, then with $\del=4/283$ and $\tet=223/249$, one has instead
$$\max\{E_+(B),E_\times(C)\}\ll |A|^3(|A|/p)^\del (\log |A|)^\tet.$$
\end{theorem}

Instead, we follow an alternative strategy in which the roles of addition and multiplication in 
our previous argument are interchanged. We begin by recording an appropriate analogue 
of Lemma \ref{lemma3.3}.

\begin{lemma}\label{lemma4.3} Let $p$ be a prime number. Then there are positive 
constants $c_1$ and $c_2$ with the following property. Suppose that $N$ is a sufficiently 
large natural number. Let $B$ be a non-empty finite subset of $\dbF_p$ with $|B|\le N$, 
and suppose that $\del$ and $\tet$ are real numbers with $0<\del<1$. Then, either 
$E_\times (B)\le 2N^{3-\delta}(\log N)^\tet$, or else there exists a subset $A'$ of $B$ 
such that
$$|A'|\ge c_1N^{1-3\delta/4}(\log N)^{(3\tet-5)/4}\quad \text{and}\quad 
|A'\cdot A'|\le c_2N^{7\delta}(\log N)^{5-7\tet}|A'|.$$
\end{lemma}

\begin{proof} Let $g\in \dbF_p^\times$ be a primitive root, and put
$$I=\{ r\in [1,p-1]:g^r\in B\}.$$
Taking account of the possibility that $0\in B$, we see that
$$E_\times (B)\le E_+(I)+4|B|^2.$$
Consequently, if $E_\times (B)>2N^{3-\del}(\log N)^\tet$, then $E_+(I)>N^{3-\del}
(\log N)^\tet$. We therefore deduce from Lemma \ref{lemma3.3} that there exists a 
subset $I'$ of $I$ such that
$$|I'|\ge c_1N^{1-3\del/4}(\log N)^{(3\tet-5)/4}\quad \text{and}\quad 
|I'+I'|\le c_2N^{7\del}(\log N)^{5-7\tet}|I'|.$$
Putting $A'=\{g^r:r\in I'\}$, the conclusion of the lemma follows.
\end{proof}

An appropriate analogue of Lemma \ref{lemma3.4} follows by applying Cauchy's 
inequality, just as before.

\begin{lemma}\label{lemma4.4} Let $A_j$ $(1\le j\le J)$ be finite subsets of a ring. Then 
one has
$$E_+\biggl( \bigcup_{j=1}^JA_j\biggr)\le J^3\sum_{j=1}^JE_+(A_j).$$
\end{lemma}

Finally, before extracting a relation between the additive energy of a set and its 
corresponding product set, we recall a consequence of a lemma from recent work of 
Roche-Newton, Rudnev and Shkredov \cite{RNRS2014} that has its origins in a paper of 
Rudnev \cite{Rud2014} concerning incidences between planes and points in three 
dimensions.

\begin{lemma}\label{lemma4.5} Let $A, B, C\subseteq \dbF_p$. Suppose that 
$|A||B||B\cdot C|\le p^2$. Then
\begin{equation}\label{4.w1}
E_+(A,C)\ll \left( |A||B\cdot C|\right)^{3/2}|B|^{-1/2}+|A||B\cdot C||B|^{-1}\max 
\{ |A|, |B\cdot C|\}.
\end{equation}
\end{lemma}

\begin{proof} This is an immediate consequence of \cite[Theorem 6]{RNRS2014}.
\end{proof}

We extract from this lemma upper bounds for $E_+(A)$ suitable for our subsequent 
applications.

\begin{lemma}\label{lemma4.6} Suppose that $A\subset \dbF_p$. Then
$$E_+(A)\ll |A\cdot A|^{3/2}|A|+p^{-1}|A\cdot A|^2|A|^2.$$
\end{lemma}

\begin{proof} Provided that $|A|^2|A\cdot A|\le p^2$, the desired conclusion follows by 
applying Lemma \ref{lemma4.5} with $B=C=A$. We have only to note that 
$|A\cdot A|\le |A|^2$, so that the second term on the right hand side of (\ref{4.w1}) is 
majorised by the first.\par

Suppose next that $|A|^2|A\cdot A|>p^2$. Put
$$n=\left\lfloor \frac{p^2}{|A||A\cdot A|}\right\rfloor. $$
Since $|A|\le p$ and $|A\cdot A|\le p$, we may suppose that $1\le n<|A|$. We take $B$ to 
be any subset of $A$ having $n$ elements. Then we have $|A||A\cdot A||B|\le p^2$. By 
applying Lemma \ref{lemma4.5} with $C=A$, we find that
$$E_+(A)\ll \left( |A||A\cdot B|\right)^{3/2}|B|^{-1/2}+|A||A\cdot B||B|^{-1}\max 
\{|A|,|A\cdot B|\}.$$
But since $B\subseteq A$, one has $A\cdot B\subseteq A\cdot A$, and hence
$$|A|\le |A\cdot B|\le |A\cdot A|\le p.$$
Moreover, since $|B|^{-1}=n^{-1}\ll |A||A\cdot A|p^{-2}$, we obtain
\begin{align*}
E_+(A)&\ll \left( |A||A\cdot A|\right)^{3/2}\left( |A||A\cdot A|p^{-2}\right)^{1/2}+
|A||A\cdot A|^2\left( |A||A\cdot A|p^{-2}\right)\\
&\ll |A|^2|A\cdot A|^2p^{-1}.
\end{align*}
This completes the proof of the lemma.
\end{proof}

We now outline the proof of our low-energy decomposition theorem over $\dbF_p$.

\begin{proof}[The proof of Theorem 1.3] We proceed precisely as in the proof of 
Theorem \ref{theorem1.1}, save that the roles of addition and multiplication are 
interchanged. For the sake of concision, we are expedient in implicity employing the 
appropriate analogue of all notation used therein. However, we provide essentially complete 
details of the argument. In the current situation, should one have 
$E_\times (A)\le 2N^{3-\delta}(\log N)^\tet$, then $A=\emptyset \cup A$ is trivially a 
decomposition of the type we seek, and so we may suppose henceforth that 
$E_\times (A)>2N^{3-\delta}(\log N)^\tet$. We proceed inductively to define subsets 
$A_j$ of $A$ for $1\le j\le K$, for a suitable integer $K$. Suppose that $k\ge 0$ and that 
the first $k$ of these sets have been defined. Put
\begin{equation}\label{5.0}
B_k=\bigcup_{j=1}^kA_j\quad \text{and}\quad C_k=A\setminus B_k.
\end{equation}
Should $E_\times (C_k)\le 2N^{3-\delta}(\log N)^\tet$, then we set $K=k$ and stop. 
Otherwise, we define the set $A_{k+1}$ as follows. We may suppose that 
$E_\times (C_k)>2N^{3-\delta}(\log N)^\tet$, and so it follows from Lemma 
\ref{lemma4.3} that there exists $A_{k+1}\subseteq C_k\subseteq A$ such that
$$|A_{k+1}|\ge c_1N^{1-3\del/4}(\log N)^{(3\tet-5)/4}$$
and
\begin{equation}\label{5.2} 
|A_{k+1}\cdot A_{k+1}|\le c_2N^{7\del}(\log N)^{5-7\tet}|A_{k+1}|.
\end{equation}
Having defined the set $A_{k+1}$, we may define $B_{k+1}$ and $C_{k+1}$ according 
to (\ref{5.0}), and repeat this decomposition argument.\par

As in the corresponding proof of Theorem \ref{theorem1.1} in \S3, the iteration defined in 
the last paragraph must terminate for some 
$K\le \lfloor c_1^{-1}N^{3\del/4}(\log N)^{(5-3\tet)/4}\rfloor$. We again put $B=B_K$ 
and $C=C_K$, and note that $A$ is the disjoint union of $B$ and $C$. In particular, we now 
have $E_\times (C)\le 2N^{3-\del}(\log N)^\tet$. We again group the subsets $A_j$ by 
cardinality, taking $\grA_m$ to be the union of those subsets $A_j$ for which 
$2^{-m}N<|A_j|\le 2^{1-m}N$. The application of Lemma \ref{lemma4.4} replaces 
(\ref{4.wa}) by the estimate
$$E_+(B)\ll (\log N)^3\sum_k\sum_{A_i\subseteq \grA_k}2^{3k}E_+(A_i).$$
By applying Lemma \ref{lemma4.6}, we obtain the bound
$$E_+(B)\ll (\log N)^3\sum_k\sum_{A_i\subseteq \grA_k}2^{3k}\left( 
|A_i\cdot A_i|^{3/2}|A_i|+p^{-1}|A_i\cdot A_i|^2|A_i|^2\right).$$
Next, on applying the property (\ref{5.2}) of these subsets $A_i$, we infer that
$$E_+(B)\ll \left(N^\del(\log N)^{1-\tet}\right)^{21/2}\sum_k 2^{3k}
\sum_{A_i\subseteq \grA_k}\biggl( |A_i|^5+p^{-2}N^{7\del}(\log N)^{5-7\tet}
|A_i|^8\biggr)^{1/2}.$$
The inner sum here is empty whenever (\ref{4.ww2}) holds. Note also that the definition of 
$\grA_k$ ensures that when $A_i\subseteq \grA_k$, then $2^k|A_i|\le 2N$. Then on 
recalling (\ref{4.ww1}), we deduce that
\begin{align}
E_+(B)\ll &\, N^{(3+21\del)/2}(\log N)^{21(1-\tet)/2}\left( N^{3\del/4}
(\log N)^{(5-3\tet)/4}\right)^{3/2}N\notag \\
&\ \ \ \ \ \ \ \ \ \ \ \ \ \ \ \ \ \ \ \ \ \ \ \ \ \ +p^{-1}N^{3+14\del}
(\log N)^{13-14\tet}N\log N\notag \\
\ll &\, \left( N^{20+93\del}(\log N)^{99-93\tet}\right)^{1/8}+p^{-1}N^{4+14\del}
(\log N)^{14(1-\tet)}.\label{r.0}
\end{align}

\par Recall the definitions of $\alp$ and $\bet$ from the statement of Theorem 
\ref{theorem1.3}. We take $\ome$ and $\nu$ to be any real numbers with 
$p=\frac{1}{10}N^\ome (\log N)^\nu$. In the first instance, we constrain our choices of 
$\del$ and $\tet$ to satisfy the inequalities
\begin{equation}\label{r.1}
\del\le \frac{8\ome -12}{19}\quad \text{and}\quad \tet\ge \frac{13-8\nu}{19}.
\end{equation}
In such circumstances, one finds that it is the first term on the right hand side of (\ref{r.0}) 
that dominates. We consequently define $\del$ and $\tet$ by means of the equations 
$8(3-\del)=20+93\del$ and $8\tet=99-93\tet$, which is to say that $\del =4/101$, and 
$\tet=99/101$. These choices for $\del$ and $\tet$ satisfy the constraint (\ref{r.1}) 
provided that
$$\ome\ge \frac{19\del+12}{8}=\frac{161}{101}\quad \text{and}\quad 
\nu\ge \frac{13-19\tet}{8}=-\frac{71}{101},$$
and this may be assured when $p\ge \frac{1}{10}N^{1/\alp}(\log N)^{-\bet/\alp}$. This 
latter constraint is satisfied when $N\le p^\alp (\log p)^\bet$. Under the latter condition, 
therefore, we obtain the bound
$$\max\{ E_+(B),E_\times (C)\}\ll N^{3-\del}(\log N)^{1-\del/2}.$$
Since $N=|A|$, this confirms the first conclusion of Theorem \ref{theorem1.3}.\par

We next take $\ome$ and $\nu$ to be any real numbers with $p=10N^\ome (\log N)^\nu$. 
In this second instance, we constrain our choices of $\del$ and $\tet$ to satisfy the 
inequalities
\begin{equation}\label{r.2}
\del\ge \frac{8\ome -12}{19}\quad \text{and}\quad \tet\le \frac{13-8\nu}{19}.
\end{equation}
In such circumstances, one finds that it is the second term on the right hand side of (\ref{r.0}) 
that dominates. We consequently define $\del$ and $\tet$ by means of the equations 
$3-\del=4+14\del-\ome$ and $\tet=14(1-\tet)-\nu$, which is to say that 
$\del =(\ome-1)/15$ and $\tet=(14-\nu)/15$. These choices for $\del$ and $\tet$ satisfy the 
constraint (\ref{r.2}) provided that
$$\ome\le \frac{19\del+12}{8}=\frac{19\ome+161}{120}\quad \text{and}\quad 
\nu\le \frac{13-19\tet}{8}=\frac{19\nu-71}{120}.$$
These conditions are satisfied provided that
$$\ome\le \frac{161}{101}\quad \text{and}\quad \nu\le -\frac{71}{101},$$
and this may be assured when $p\le 10N^{1/\alp}(\log N)^{-\bet/\alp}$. This latter 
constraint is satisfied when $N\ge p^\alp (\log p)^\bet$. Under the latter condition, 
therefore, we obtain the bound
$$\max\{ E_+(B),E_\times (C)\}\ll N^{3-\del}(\log N)^\tet 
=N^3(N/p)^{1/15}(\log N)^{14/15}.$$
Since $N=|A|$, this confirms the second conclusion of Theorem \ref{theorem1.3}, and 
completes the proof of Theorem \ref{theorem1.3}.
\end{proof}

We finish this section by confirming the lower bounds (\ref{1.x}) presented following 
Theorem \ref{theorem1.3}. This is a simple exercise in exponential sums over finite fields, 
in which we employ the standard notation $e_p(x)=e^{2\pi ix/p}$. When 
$B,C\subseteq \dbF_p$, write
$$f(u)=\sum_{b\in B}e_p(ub)\quad \text{and}\quad g(u)=\sum_{c_1,c_2\in C}
e_p(uc_1c_2).$$
Then, by orthogonality, one finds that
$$E_+(B)=p^{-1}\sum_{u=0}^{p-1}|f(u)|^4\quad \text{and}\quad 
E_\times (C)=p^{-1}\sum_{u=0}^{p-1}|g(u)|^2.$$
Using positivity, and discarding all terms in each sum save for that with $u=0$, we 
thus conclude that
$$E_+(B)\ge p^{-1}f(0)^4=p^{-1}|B|^4\quad \text{and}\quad 
E_\times(C)=p^{-1}g(0)^2=p^{-1}|C|^4.$$
This confirms the desired lower bounds.

\section{Higher order energies} We finish our account of low-energy decomposition 
theorems with a brief discussion of higher order energies, and in particular the proof of 
Theorem \ref{theorem1.4}. With $\bfb$ as shorthand for $(b_3,\ldots ,b_k)$, write
$$T(\bfa,\bfa')=\text{card}\{(a_1,a_2),(a'_1,a'_2)\in A_1\times 
A_2:a_1+\dots+a_k=a_1'+\dots+a_k'\}.$$
For $k\geq 2$, one obtains cheap bounds on the $k$-fold additive energy between sets 
$A_1,\ldots ,A_k$ by means of the relation
\begin{align*}
E_+(A_1,\dots,A_k)&=\sum_{\bfa,\bfa'\in A_3\times \ldots \times A_k}T(\bfa,\bfa')\\
&\le |A_3|^2\cdots|A_k|^2\max_{b}U(b),
\end{align*}
where we write
$$U(b)=\text{card}\{(a_1,a_2),(a'_1,a'_2)\in A_1\times A_2:a_1+a_2+b=a_1'+a_2'\}.$$
By Cauchy's inequality, for any fixed value of $b$, one has
\begin{align*}
U(b)&=\sum_{x\in A_1+A_2}r_{A_1+A_2}(x-b)r_{A_1+A_2}(x)\\
&\le \biggl(\sum_{x\in A_1+A_2-b}r_{A_1+A_2}(x)^2\biggr)^{1/2}
\biggl( \sum_{x\in A_1+A_2} r_{A_1+A_2}(x)^2\biggr)^{1/2}\\
&\le E_+(A_1,A_2).
\end{align*}
Thus we obtain the bound 
$$E_+(A_1,\dots,A_k)\le |A_3|^2\cdots|A_k|^2E_+(A_1,A_2).$$
An entirely analogous argument coughs up the corresponding bound
$$E_\times (A_1,\dots,A_k)\le |A_3|^2\cdots|A_k|^2E_\times (A_1,A_2).$$

\par Given a finite subset $A$ of the real numbers, consider positive integers $m$ and $n$ 
with $m\ge 2$ and $n\ge 2$. With $\del=\frac{2}{33}$, it follows from Theorem 
\ref{theorem1.1} that there exist disjoint subsets $B$ and $C$ of $A$, with $A=B\cup C$, 
and
$$\max\{ E_+(B),E_\times (C)\} \ll |A|^{3-\del}(\log |A|)^{1-\del}.$$
Using these same subsets $B$ and $C$, it follows from our opening discussion in this 
section that
$$E_+^{(m)}(B)\le |A|^{2m-4}E_+(B)\ll |A|^{2m-1-\del}(\log |A|)^{1-\del},$$
and likewise
$$E_\times^{(n)}(C)\le |A|^{2n-4}E_\times (B)\ll |A|^{2n-1-\del}(\log |A|)^{1-\del}.$$
This completes the proof of Theorem \ref{theorem1.4}.

\bibliographystyle{amsbracket}
\providecommand{\bysame}{\leavevmode\hbox to3em{\hrulefill}\thinspace}

\end{document}